\documentclass[a4paper,11pt]{amsart}
\usepackage[plainpages=false]{hyperref}
\usepackage{amsfonts,latexsym,rawfonts,amsmath,amssymb,amsthm, mathrsfs, lscape, calligra}
\usepackage{verbatim}
\usepackage{enumerate}
\usepackage{centernot}
\usepackage{mathtools}
\usepackage{tikz}
\usepackage{tikz-cd}

\usepackage[all]{xy}
\usepackage{graphicx,psfrag}

\usepackage{array, tabularx, color}

\usepackage{setspace}

\newtheorem{thm}{Theorem}[section]
\newtheorem{cor}[thm]{Corollary}
\newtheorem{lem}[thm]{Lemma}

\newtheorem{prop}[thm]{Proposition}

\theoremstyle{remark}
\newtheorem{rmk}[thm]{Remark}

\theoremstyle{definition}
\newtheorem{defi}[thm]{Definition}

\def \R {\mathbb{R}}

\def \C {\mathbb{C}}
\def \E {\mathcal{E}}

\def \db {\bar \partial}

\DeclareMathOperator\dVol {dVol}

\DeclareMathOperator \Id {Id}

\newcommand{\CBbb}{\mathbb C}

\newcommand{\ZBbb}{\mathbb Z}

\newcommand{\Ecal}{\mathcal E}

\newcommand{\Scal}{\mathcal S}

\newcommand{\GL}{\mathsf{GL}}

\newcommand{\U}{\mathsf{U}}

\newcommand{\SO}{\mathsf{SO}}

\DeclareMathOperator{\End}{End}

\DeclareMathOperator{\ch}{ch}

\DeclareMathOperator{\rank}{rank}

\DeclareMathOperator{\vol}{vol}

\DeclareMathOperator{\tr}{tr}

\DeclareMathOperator{\Iso}{Iso}

\newcommand{\dbar}{\bar\partial}

\numberwithin{equation}{section}

\makeatletter
\@namedef{subjclassname@2020}{\textup{2020} Mathematics Subject Classification}
\makeatother

\begin{document}

\title[Nonabelian Hodge correspondence]{The nonabelian Hodge 
correspondence for balanced Hermitian metrics of Hodge-Riemann type}
\author[Chen]{Xuemiao Chen}
\author[Wentworth]{Richard A. Wentworth}
\thanks{R.W.'s research is supported in part by NSF grant DMS-1906403.}
\address{Department of Mathematics, University of Maryland, College Park,
MD 20742, USA}
\email{xmchen@umd.edu}
\email{raw@umd.edu}

\subjclass[2020]{Primary: 53C07, 14J60; Secondary: 58E20}
\keywords{Hermitian-Einstein metric, Higgs bundle, harmonic map}

\begin{abstract}
This paper extends the nonabelian Hodge correspondence for
    K\"ahler manifolds to a larger class of hermitian metrics on
    complex manifolds called balanced 
     of Hodge-Riemann type. Essentially, it grows out of a few key observations so that the known results, especially the Donaldson-Uhlenbeck-Yau theorem and Corlette's theorem, can be applied in our setting.    
    Though not necessarily K\"ahler, we show that the Sampson-Siu Theorem 
    proving that harmonic maps are pluriharmonic remains valid for a slightly smaller class by using the known argument. Special important examples 
    include those balanced metrics arising  from multipolarizations. 
\end{abstract}
\maketitle
\thispagestyle{empty}
\bibliographystyle{amsplain}

\section{Introduction}
Let $X$ be a compact, complex manifold of dimension $n$.
Recall that a hermitian metric on
$X$ is called \emph{balanced} if 
$d\omega^{n-1}=0$, where $\omega$ is the fundamental (K\"ahler)
$(1,1)$-form of the metric.
The balanced metrics are a more restrictive
class than the \emph{Gauduchon} metrics, which satisfy
$\partial\dbar\omega^{n-1}=0$.
Nevertheless, there are many examples of balanced, non-K\"ahler, metrics
(cf.\ \cite[p.\ 292]{Michelsohn:82}).

In this paper we consider a further condition.
We say that a balanced  metric is of \emph{Hodge-Riemann type}, if it
admits an expression: 
\begin{align} \label{HRtype}
    \frac{\omega^{n-1}}{(n-1)!}=\omega_0\wedge\Omega_0
\end{align}
where $\omega_0$ (resp.\ $\Omega_0$) is a real $(1,1)$
(resp.\ $(n-2,n-2)$), and  $\Omega_0$ satisfies the Hodge-Riemann bilinear
relations (see Definition \ref{Defi:Hodge-Riemann} for the precise
definition). 

The condition of being balanced of Hodge-Riemann type seems very restrictive.
However,  many examples of non-K\"ahler metrics satisfying this
property come from \emph{multipolarizations}. 
Namely, let $\omega_0,\omega_1,  \ldots, \omega_{n-2}$
be positive $(1,1)$-forms on $X$, and suppose
\begin{equation} \label{eqn:multipolarization}
\frac{\omega^{n-1}}{(n-1)!}=\omega_0 \wedge \cdots \wedge \omega_{n-2}
\end{equation}
such that
\begin{itemize}
\item $d\omega^{n-1}=0$;
\item $d(\omega_1 \wedge \cdots\wedge \omega_{n-2})=0$
\end{itemize}
(e.g.\ both conditions are automatic if the $\omega_i$ are K\"ahler).
Then  by a
result of Timorin \cite{Timorin:98}, $\omega$ is of Hodge-Riemann type. 
Even when the $\omega_i$ are K\"ahler metrics, $\omega$ is not so in general.

In this note we show that certain properties of Hermitian-Einstein metrics and
equivariant harmonic maps  familiar for K\"ahler manifolds continue to hold
for balanced metrics of Hodge-Riemann type. Namely, we prove
\begin{enumerate}
    \item a generalized Bogomolov-Miyaoka-Yau inequality for $\omega$-polystable
        holomorphic (and Higgs) bundles (Corollary \ref{cor:bmy});
    \item a version of the Sampson-Siu pluriharmonicity theorem  for harmonic maps
        to targets with nonpositive complexified sectional curvature (Theorem
        \ref{thm:siu-sampson} and Corollary \ref{cor:siu-sampson});
    \item the nonabelian Hodge correspondence relating $\omega$-stable
        Higgs bundles with vanishing Chern classes to irreducible
        representations of the fundamental group (Theorem \ref{thm:nah}).
\end{enumerate}
Let us remark that for the class of Gauduchon metrics, items (2) and (3) do \emph{not}
hold in general (see \cite{Biswas:11}), and the statement of item (1) cannot even be formulated.

The simple idea behind these generalizations is well explained in
\cite[Lemma 1.1]{Simpson:92}. Let us focus on item (3) above.  Suppose $D$ be a 
complex connection on a vector bundle $E\to X$.
A hermitian metric $h$ on $E$ gives a decomposition $D=D''+D'$
(see \eqref{eqn:hodge}).  Conversely, a Higgs bundle defines an operator $D''$, and a
metric allows one to complete it to a complex connection $D$  by setting
$D'=(D'')^\ast$. 
Let $F_D=D^2$ be the curvature of $D$, 
and $G_D=(D'')^2$ the \emph{pseudo-curvature}. 
Flatness of $D$ is the equation $F_D=0$, whereas $D$ arises from a Higgs
bundle iff $G_D=0$.  

Now suppose $\omega$ is a balanced metric on $X$, so that the degree and slope
stability of holomorphic bundles can be defined. 
Given  an $\omega$-slope stable Higgs bundle with $\ch_1(E)=0$,
one can find a metric $h$ so
that the associated connection  $D$ satisfies
\begin{equation} \label{eqn:F}
    F_D\wedge \omega_0\wedge\Omega_0=0
\end{equation}
 Similarly, given a flat connection $D$
one can find a \emph{harmonic metric}, meaning that 
\begin{equation} \label{eqn:G}
G_D\wedge\omega_0\wedge\Omega_0=0
\end{equation} 
Thus, under the assumptions,  the forms $F_D$ and $G_D$ are ``primitive'',
in the sense of \eqref{eqn:primitive} below. 

The nonabelian Hodge correspondence follows by showing that if
in addition
$\ch_2(E)=0$, then \eqref{eqn:F} implies $F_D=0$, and on the other hand,
\eqref{eqn:G} always implies
$G_D=0$  (as pointed out in \cite[p.\ 17]{Simpson:92}, the ``pseudo-Chern
class'' defined by $G_D$ automatically vanishes by the flatness of $D$). 
Now, if we assume $\omega$ is of Hodge-Riemann type, then these conclusions hold
by integrating $\tr(F_D\wedge F_D)$ or $\tr(G_D\wedge G_D)$ 
against $\Omega_0$, and using the vanishing of the Chern classes and the
Hodge-Riemann bilinear relations. 
Thus, we see that the K\"ahler condition may be relaxed.

\medskip
\noindent
{\bf Acknowledgements}. 
The authors warmly thank
Yang Li for comments on an earlier version of this paper.

\section{Hodge-Riemann forms}
In this section, we recall the notion of  a Hodge-Riemann form on a
polarized complex vector space, i.e.\ a complex space $V$ with a constant 
K\"ahler form $\omega_0$. Denote $\Lambda^{p,q}$ to be the space of constant $(p,q)$ forms over $V$. We fix $\Omega_0$ to be any real $(n-p-q,n-p-q)$ form. On $\Lambda^{p,q}$, we can define a Hermitian form as 
$$
Q(\alpha, \beta):=(\sqrt{-1})^{p-q} (-1)^{\frac{(p+q)(p+q-1)}{2}}
*(\alpha \wedge \overline{\beta}\wedge \Omega_0)
$$
The space of \emph{primitive} forms of degree $(p,q)$ associated to $(\Omega_0, \omega_0)$ is defined as
\begin{equation} \label{eqn:primitive}
P^{p,q}=\{\alpha\in \Lambda^{p,q}: \alpha \wedge \omega_0 \wedge \Omega_0= 0\}.
\end{equation}
\begin{defi}\label{Defi:Hodge-Riemann}
We call $\Omega_0$ a Hodge-Riemann form for degree $(p,q)$ with respect to $\omega_0$ if 
\begin{enumerate}
\item there exists a $Q$-orthogonal decomposition 
$$\Lambda^{p,q}=\C \omega_0\oplus P^{p,q};$$
\item $Q$ is positive definite on $P^{p,q}$. 
\end{enumerate}
\end{defi}

\begin{rmk}
It follows from the classical Hodge-Riemann relation that
    $\Omega_0=\omega_0^{n-p-q}$ is a Hodge-Riemann form (cf.\ \cite[Thm.\
    6.32]{Voisin:07}). 
\end{rmk}

In general, the Hodge-Riemann property of a form is difficult to verify.
However, 
we have the following result,  which has been used to get the
mixed Hodge-Riemann relation (see \cite{DinhNguyen:06}).
\begin{prop}[{\cite[Main Theorem]{Timorin:98}, see also \cite{Gromov:90}}]\label{prop1.2}
For any constant positive $(1,1)$ forms $ \omega_1,\cdots, \omega_k$ on
    $(V,\omega_0)$, $\omega_1\wedge \cdots\wedge \omega_{n-k}$
    is a Hodge-Riemann form  with respect to $\omega_0$ 
    for any degrees $(p,q)$ satisfying $p+q=k$.
\end{prop}

As a special case, this gives 
\begin{cor}\label{cor1.4}
For any constant positive $(1,1)$ forms $\omega_1, \cdots, \omega_{n-2}$ on $(V, \omega_0)$
\begin{enumerate}
\item $\omega_0 \wedge \cdots\wedge \omega_{n-2}$ is a 
    strictly positive $(n-1,n-1)$ form. In particular, there exists a positive $(1,1)$ form $\omega$ so that 
$$
        \frac{\omega^{n-1}}{(n-1)!}=\omega_0\wedge \cdots\wedge \omega_{n-2}
$$
        (cf.\ \cite[p.\ 279]{Michelsohn:82}, and also \cite{Toma:10}).
\item $\omega_1\wedge \cdots\wedge \omega_{n-2}$ is a 
    Hodge-Riemann form for degrees $(p,q)$ satisfying $p+q=2$ with respect to $\omega_0$. 
\end{enumerate}
\end{cor}

This combined with   \cite[Cor.\ 8.5]{RossToma:19} implies the following
\begin{prop}
For any constant K\"ahler forms $\omega_1,\omega_2$ on $(V, \omega_0)$, the form
$$\omega_1^{n-2}+\omega_1\wedge \omega_2^{n-3}+\cdots \omega_2^{n-2}$$ 
is a Hodge-Riemann form for degree $(p,q)$ with $p+q=2$ with respect to $\omega_0$. 
\end{prop}

\begin{rmk}
It is known that the Hodge-Riemann property is not invariant under convex
    linear combinations. For example, fix any two K\"ahler forms $\omega_1$
    and $\omega_2$ on $\C^4$, $\omega_1^2+a\omega_2^2$ is not a
    Hodge-Riemann form for degree $(1,1)$ for certain positive values of
    $a$ (see  \cite[Rem.\ 9.3]{RossToma:19} and also 
    \cite[Rem.\ 3]{Timorin:98} for other examples).
\end{rmk}

Timorin's result motivates the following:
\begin{defi} \label{def:bhr}
    A hermitian metric $\omega$ on a complex manifold $X$ is said to be 
    \emph{balanced of Hodge-Riemann type} if the following hold:
    \begin{enumerate}
        \item we have an expression
$$
\frac{\omega^{n-1}}{(n-1)!}=\omega_0 \wedge \Omega_0
$$
            where  $\omega_0$ is a strongly positive real $(1,1)$  form on
            $X$ and  $\Omega_0$ is a real $(n-2,n-2)$;
        \item at every point $\Omega_0$ is a  Hodge-Riemann form for $(p,q)$, $p+q=2$;
        \item $\Omega_0$ and $\omega_0 \wedge \Omega_0$ are closed.
    \end{enumerate}
\end{defi}
Note that (3) is  equivalent to $\omega$ being balanced and $\Omega_0$ being closed.

\section{Bogomolov-Miyaoka-Yau inequality}
Below we show how
 the Donaldson-Uhlenbeck-Yau (resp.\ Hitchin-Simpson)
theorem relating stability of holomorphic (resp.\ Higgs) bundles to the
existence of Hermitian-Einstein type metrics results in a Chern class
inequality.  The main result is Corollary \ref{cor:bmy}. 
In this section, assume $(X,\omega)$ is a compact complex Hermitian manifold 
that satisfies items (1) and (3) of Definition \ref{def:bhr}, as well as
the Hodge-Riemann condition (2) for the case $(p,q)=(1,1)$. 

Recall that associated to  every coherent analytic  sheaf $\Ecal\to X$ 
is a holomorphic line bundle $\det\Ecal$.  The first Chern class of $\Ecal$ is
by definition
$c_1(\Ecal):=c_1(\det\Ecal)\in H^2(X,\ZBbb)\cap H^{1,1}_{\dbar}(X)$.  
We define the $\omega$-degree of $\Ecal$ by:
$$
\deg\Ecal:=\int_X c_1(\Ecal)\wedge\frac{\omega^{n-1}}{(n-1)!}=\int_X
c_1(\Ecal)\wedge\omega_0\wedge\Omega_0
$$
Because of the balanced condition, this is well-defined on the class of
$c_1(\Ecal)$.  The \emph{slope} of a (nonzero) torsion-free 
sheaf is 
$$
\mu(\Ecal)=\frac{\deg\Ecal}{\rank\Ecal}
$$
Then we say a holomorphic bundle $\Ecal\to X$ is \emph{$\omega$-stable}
 if $\mu(\Scal)<\mu(\Ecal)$ for every coherent subsheaf $\Scal\subset\Ecal$
 with $0<\rank\Scal<\rank\Ecal$.

A  \emph{Higgs bundle} 
on $X$ is a pair  $(\Ecal,\theta)$, where $\Ecal\to X$ is a holomorphic
bundle, $\theta$ is a holomorphic $1$-form with values in $\End\Ecal$, and
$\theta\wedge\theta=0$. 
We say that a Higgs bundle is $\omega$-stable
 if $\mu(\Scal)<\mu(\Ecal)$ for every coherent subsheaf $\Scal\subset\Ecal$
 with $0<\rank\Scal<\rank\Ecal$ \emph{and} $\theta(\Scal)\subset
 \Scal\otimes\Omega^1_X$, where $\Omega^1_X$ is the holomorphic cotangent
 sheaf of $X$. Thus,  stable vector bundles are a special case of stable
 Higgs bundles, where $\theta\equiv 0$. 
Finally, we say that $(\Ecal,\theta)$ is \emph{$\omega$-polystable} if
$(\Ecal,\theta)$ splits as a direct sum of Higgs subbundles, all with the
same slope.

Given a hermitian metric $h$ on $\Ecal$, let 
$\dbar_E+\partial_E$ denote the Chern connection of $(\Ecal,h)$, 
$\theta^\ast$ the hermitian adjoint of $\theta$ with respect to $h$. Thus,
$\theta^\ast$ is a $(0,1)$-form with values in $\End E$, satisfying
$\partial_E\theta^\ast=0$. We will consider the \emph{complex} connection
$$
D=\dbar_E+\partial_E+\theta+\theta^\ast
$$
and its curvature $F_D$. To be explicit, we will write:
$D=(\Ecal,\theta,h)$.  

\begin{defi} \label{def:HE}
    Fix a Higgs bundle $(\Ecal,\theta)$ on $X$.
    A hermitian metric $h$  on $\Ecal$ is called Hermitian-Einstein (HE) if
\begin{equation}\label{eqn:HE}
    \sqrt{-1} F_{(\Ecal,\theta, h)}\wedge \omega_0 \wedge \Omega_0 =
    \lambda\cdot \Id\cdot \omega_0^2 \wedge \Omega_0
\end{equation} 
for some constant $\lambda$. 
\end{defi}

Now we have the following generalized
Donaldson-Uhlenbeck-Yau, Hitchin-Simpson theorem. 
\begin{thm} \label{thm:HE}
$(\Ecal, \theta)$ is $\omega$-polystable if and only if it admits a HE
    metric. Moreover, if $(\Ecal,\theta)$ is $\omega$-stable, such a metric is unique up to scaling.
\end{thm}\label{thm1.1}

We will use the key fact that for balanced metrics, the 
K\"ahler identities hold for $(1,0)$ and $(0,1)$
forms  (\cite[Prop.\ 1]{Gauduchon:77}; see also  
 \cite[Lemma 7.1.1]{LubkeTeleman:95}). 

\begin{lem}  \label{Lemma:Kahler identity}
Given an $n$-dimensional hermitian manifold $(X, \omega)$ with
    $d\omega^{n-1}=0$, the following hold:
$$
\db^*\alpha^{0,1}=-\sqrt{-1}\Lambda\partial\alpha^{0,1}\quad , \quad
    \partial^*\alpha^{1,0}=\sqrt{-1}\Lambda\dbar\alpha^{1,0}
$$
for any $(0,1)$-form $\alpha^{0,1}$, and any $(1,0)$-form $\alpha^{1,0}$.
\end{lem} 

\begin{proof}[Proof of Theorem \ref{thm:HE}]
That the existence of a HE metric implies polystability is well-known.
    For the converse, it suffices to  assume $(\Ecal,\theta)$ is
    $\omega$-stable. 
Since  $\omega$ is balanced, it is in particular
Gauduchon, and so by the result of Li-Yau \cite{LiYau:87}, generalized to
    Higgs bundles by L\"ubke-Teleman \cite{LubkeTeleman:06}, 
    there is a metric $\widetilde h$ such that
$$
    \sqrt{-1} F_{(\Ecal,\theta,\widetilde h)}\wedge \frac{\omega^{n-1}}{(n-1)!} =
    \widetilde\lambda\cdot \Id\cdot \frac{\omega^{n}}{n!}
    $$
    where $\widetilde \lambda = 2\pi \mu(\Ecal)/\vol(X,\omega)$.
    Now there is a positive function $f$ such that
    $$
\omega_0\wedge \frac{\omega^{n-1}}{(n-1)!} = f\cdot\frac{\omega^{n}}{n!}
    $$
    Choose $\lambda$ such that
    \begin{equation} \label{eqn:lambda} 
    \lambda\int_X f\cdot \frac{\omega^{n}}{n!} =\widetilde\lambda
    \vol(X,\omega)=2\pi \mu(\Ecal)
    \end{equation} 
    Then we can find a function $\varphi$ satisfying: 
    $
    \Delta_\omega \varphi=2(\lambda f-\widetilde\lambda)
    $.
Let $h=e^\varphi\widetilde h$. Then 
$$
    F_{(\Ecal,\theta,h)}=F_{(\Ecal,\theta,\widetilde
    h)}-\partial\dbar\varphi\cdot \Id
$$
    By Lemma \ref{Lemma:Kahler identity},  the Hodge and Dolbeault
    laplacians on functions are related: $\Delta_\omega=2\Delta_{\dbar}$. 
    Hence,
    $$
    -i\partial\dbar\varphi\wedge
    \frac{\omega^{n-1}}{(n-1)!}=\frac{1}{2}\Delta
    \varphi\frac{\omega^{n}}{n!}=\lambda
    \omega_0\wedge\frac{\omega^{n-1}}{(n-1)!} -\widetilde\lambda \frac{\omega^{n}}{n!}  
    $$
    The result follows.
\end{proof}

As a direct corollary of this, we have the following generalized Bogomolov-Miyaoka-Yau inequality 
\begin{cor}\label{cor:bmy}
For any $\omega$-polystable rank $r$ Higgs bundle $(\E, \theta)$,  the following
    inequality holds:
$$
\int_X (2r c_2(\E)-(r-1) c_1(\E)^2) \wedge \Omega_0 \geq 0
$$
where the equality holds if and only if $\E$ is projectively flat. 
\end{cor}

\begin{proof}
    Eq.\ \eqref{eqn:HE} implies
    $$
    F_{(\Ecal,\theta,h)}-\frac{1}{r}\tr(F_{(\Ecal,\theta,h)})\cdot\Id\cdot\, \omega_0
    $$
    is primitive. Now use the Hodge-Riemann property of $\Omega_0$. 

\end{proof}

For emphasis, we state the following
 version of the Donaldson-Uhlenbeck-Yau theorem for the 
 slope stability condition defined by  multipolarizations (cf.\
\cite{GrebToma:17}).  

\begin{thm}
Suppose $X$ is a compact K\"ahler manifold with $(n-1)$ K\"ahler forms 
    $\omega_0, \cdots, \omega_{n-2}$. Given a  holomorphic
    vector bundle $\E$ that is slope stable  with respect to $[\omega_0]\cup\cdots\cup
    [\omega_{n-2}]$, there exists a Hermitian-Einstein metric $h$
    on $\E$, i.e.\ 
$$
    \sqrt{-1}F_{(\Ecal, h)} \wedge\omega_0\wedge \cdots\wedge \omega_{n-2} = \lambda\cdot
    \Id\cdot\,  \omega^2_0\wedge \omega_1 \wedge\cdots\wedge \omega_{n-2}
$$ 
for some constant $\lambda$. Moreover, such a metric is unique up to constant rescalings. 
\end{thm}

\begin{rmk}
    We emphasize here that the Chern connection  of $(\Ecal,h)$ 
    is \emph{not} a Yang-Mills connection in general.
\end{rmk}

By Corollary \ref{cor:bmy} and Proposition \ref{prop1.2}, we have
the following generalization of  the  Bogomolov-Gieseker inequality for
multipolarizations, proven in the projective case by Miyaoka
\cite[Cor.\ 4.7]{Miyaoka:87}.  

\begin{cor}
Suppose $X$ is a compact K\"ahler manifold with $(n-1)$ K\"ahler forms
    $\omega_0, \cdots ,\omega_{n-2}$, and $\E$ is a slope stable
    holomorphic vector bundle over $X$ with respect to $[\omega_0]\cup
    \cdots\cup[\omega_{n-2}]$. Then the following holds: 
$$
\int_X (2r c_2(\E)-(r-1) c_1^2(\E)) \wedge \Omega_j \geq 0
$$
for any $j=0,\ldots, n-2$. Here, $\Omega_j=\omega_0 \wedge \cdots\wedge \omega_{j-1} \wedge
    \omega_{j+1} \cdots\wedge \omega_{n-2}$. Moreover, 
    the equality holds for some $j$ if and only if $\E$ is projectively flat. 
\end{cor}

\begin{rmk}
The gauge theoretic side of the HE connections defined via multipolarizations 
    is studied in \cite{ChenWentworth:21a}. 
\end{rmk}

\section{The Sampson-Siu theorem}

In this section, we prove

\begin{thm} \label{thm:siu-sampson}
    Let $X$ be a compact complex manifold with a balanced metric of
    Hodge-Riemann type, and assume $\Omega_0$ is strongly positive. If $N$
    is a Riemannian manifold with nonpositive
    complexified sectional curvature, then every harmonic map $u:X\to N$ is
    pluriharmonic.
\end{thm}

In the statement of the theorem, 
$\omega$ satisfies the condition of 
 Definition \ref{def:bhr}, 
but  we make the \emph{additional assumption} that $\Omega_0$ is a \emph{strongly
positive} $(n-2,n-2)$-form in the sense of \cite[Ch.\ III, Def.\ 1.1]{Demailly}.

\begin{proof}[Proof of Theorem \ref{thm:siu-sampson}]
    Let $\nabla$ denote the Levi-Civit\`a connection on $N$. This induces a
    connection on $u^\ast TN$. The harmonic map equation is: $d_\nabla^\ast
    du=0$. Since $\omega$ is balanced, 
    Lemma \ref{Lemma:Kahler identity} implies that $u$ is harmonic if and
    only if 
    \begin{equation} \label{eqn:harmonic}  
    d_\nabla d^c u\wedge\omega^{n-1}=0
    \end{equation} 
    Next, we follow the argument in \cite[pp.\ 73-75]{ToledoBook}.
Since $\Omega_0$ is closed,
$$
    d\langle d_\nabla d^c u\wedge d^c u\rangle\wedge\Omega_0=
    \langle R_N(d^c u)\wedge d^c u\rangle\wedge\Omega_0 + \langle d_\nabla d^c u
    \wedge d_\nabla d^c u\rangle\wedge\Omega_0
$$
and so,
$$
    0=\int_X\left\{
    \langle R_N(d^c u)\wedge d^c u\rangle\wedge\Omega_0 + \langle d_\nabla d^c u
    \wedge d_\nabla d^c u\rangle\wedge\Omega_0
    \right\}
$$
    By \eqref{eqn:harmonic} and  the Hodge-Riemann property, 
    the second term is nonpositive. We claim
    that also 
    $$
\langle R_N(d^c u)\wedge d^c u\rangle\wedge\Omega_0 \leq 0
$$
Given this, it follows that $d_\nabla d^c u=0$; hence, the
    pluriharmonicity. The claim follows from the assumption on $\Omega_0$
    and the nonpositivity of the complexified sectional curvature of $N$.
    We work at a point $x$. By definition, we know that
    $$
    \Omega_0=\sum_{i} \mu_i \sqrt{-1}\alpha_1^i \wedge
    \overline{\alpha_1^i}\wedge \cdots\wedge \sqrt{-1}\alpha_{n-2}^i \wedge \overline{\alpha_{n-2}^i}
    $$
    where $\mu_i\geq 0$, and $\{\alpha_1^{i},\cdots, \alpha_{n-2}^i\}$ 
    are linearly independent $(1,0)$ forms. 
    Denote by $P_i$  the complex two dimensional subspace of $TM$ where
    $\alpha^i_j|_{P_i}=0$ for $j=1,\cdots, n-2$. 
    Fix $X_i, Y_i$ so that $\{X_i, Y_i, JX_i, JY_i\}$ form an orthogonal basis for $P_i$.
    Then 
    $$
    \langle R_N(d^c u)\wedge d^c u\rangle\wedge\Omega_0=\sum_i \mu_i'\langle R_N(d^c u)\wedge d^c u\rangle (X_i,Y_i,JX_i,JY_i)\dVol
    $$
    for some $\mu_i'\geq 0$. Now as in \cite[p.\ 75]{ToledoBook}, we know 
    $$\langle R_N(d^c u)\wedge d^c u\rangle (X_i,Y_i,JX_i,JY_i)=R_N(Z_i,W_i,\overline{W_i},\overline{Z_i}) \leq 0$$ 
    where $Z_i=du(X_i-JX_i)$ and $W_i=du(W_i-JW_i)$. The claim follows.
\end{proof}

\begin{rmk} \label{rmk:equivariant}
    If $\widetilde X$ is the universal cover of $X$, then
    Theorem \ref{thm:siu-sampson} remains valid for harmonic maps
    $u:\widetilde X\to N$ that are equivariant with respect to a
    representation $\rho:\pi_1(X)\to \Iso(N)$. These play a role in the
    next section. The existence of equivariant harmonic maps to
    nonpositively curved targets $N$ is guaranteed if $\rho$ is
    \emph{reductive} (or \emph{semisimple}) (see \cite{Corlette:88,
    Donaldson:87, Labourie:91, JostYau:91}). 
\end{rmk}

Theorem \ref{thm:siu-sampson},  
combined with \cite[Prop.\ III.1.11]{Demailly} implies
\begin{cor} \label{cor:siu-sampson}
Suppose $X$ is a compact complex manifold with a balanced metric $\omega$
    of the form \eqref{eqn:multipolarization},
  where $\omega_i$ are positive $(1,1)$-forms and 
    $$d(\omega_1 \wedge \cdots \wedge \omega_{n-2})=0$$ 
    If $N$
    is a Riemannian manifold with nonpositive
    complexified sectional curvature, then every harmonic map $u:X\to N$ is
    pluriharmonic.
\end{cor}

\section{The nonabelian Hodge correspondence}

The goal of this section is to prove the following generalization of
\cite[Cor.\ 1.3]{Simpson:92}.

\begin{thm} \label{thm:nah}
Suppose $X$ is a compact complex manifold
and $\omega$ is a balanced metric of Hodge-Riemann type on $X$. 
    Then the nonabelian Hodge correspondence holds over $(X, \omega)$. More
    precisely, we have a 1-1 correspondence between 
\begin{enumerate}
\item semisimple flat bundles on $X$,  and
\item isomorphism classes of $\omega$-polystable Higgs bundles $(\Ecal, \theta)$ 
    with $\ch_1(\Ecal)\cup[\omega^{n-1}]=0$ and $\ch_2(\Ecal)\cup[\Omega_0]=0$.
\end{enumerate}
\end{thm}

\begin{rmk}
\begin{itemize}
    \item When $\omega_0=\omega$, $\Omega_0=\omega^{n-2}/(n-1)!$, then by
        our assumptions  
        $(X, \omega)$ is K\"ahler. Then Theorem \ref{thm:nah} reduces to the
        well known nonabelian Hodge correspondence for compact K\"ahler manifolds.

\item There exist many examples where $\omega$ is not a  K\"ahler
    metric, even when the underlying manifold $X$ is K\"ahler, or even
        projective algebraic. For example, 
        take $X$ to be projective with $[\omega_i]$, $i=0,\ldots, n-2$, 
        all ample classes, and take $\omega$ as in
        \eqref{eqn:multipolarization}.
        Then  the class $\omega^{n-1}$ represents a point in the 
        interior of the cone of \emph{movable curves} (see \cite{BDPP:2013}). 

\item 
    Notice that if $\Omega_0=\omega_1^{n-2}$, then $d\Omega_0=0$ implies
        $d\omega_1=0$, and 
    the manifold is K\"ahler. However, closedness of
        $\Omega_0=\omega_1\wedge\cdots\wedge\omega_{n-2}$, for different
        $\omega_i$, can occur in the non-K\"ahler setting. 
        One might expect that this will provide new insights for the study the
        non-K\"ahler complex manifolds,
        since the results obtained here already put restrictions on complex manifolds
        admitting such structures.
\end{itemize}
\end{rmk}

\begin{proof}[Proof of Theorem \ref{thm:nah}]
The proof, of course, closely follows the lines of the classical theorem, taking
    care to avoid the K\"ahler condition. 

    First, assume $(\Ecal,\theta)$ is an $\omega$-polystable Higgs bundle.
    If $\ch_1(\Ecal)\cup[\omega^{n-1}]=0$, then by Theorem \ref{thm:HE}
    there is a hermitian metric $h$ on $\Ecal$ such that \eqref{eqn:F} is
    satisfied for 
    \begin{equation} \label{eqn:D}
    D=\dbar_E+\partial_E+\theta+\theta^\ast
    \end{equation}
    (note that $\lambda=0$ by \eqref{eqn:lambda}). 
    Hence, $F_D$ is primitive.  
    Moreover, $\sqrt{-1}F_D$ is of type $(1,1)$ and hermitian.
    Since $\ch_2(\Ecal)\cup[\Omega_0]=0$, we have
    $$
    0=\int_X\tr(F_D\wedge F_D)\wedge\Omega_0
    $$
    Since $\Omega_0$ is a Hodge-Riemann form, we conclude that $D$ is a
    flat connection, which is necessarily semisimple.

Now let $D$ be a semisimple flat connection on $E$.
    By Corlette's theorem (see Remark \ref{rmk:equivariant}) there exists a
    \emph{harmonic metric}, which has the following consequence.
    Decomposing into type we can express $D$ as in \eqref{eqn:D},
    where $\theta\in \Omega^{1,0}(X,\End E)$, and  $\theta+\theta^\ast$ is
    essentially $du$ for an equivariant harmonic map $u:\widetilde X\to
    \GL(n,\CBbb)/\U(n)$. 
Let
    \begin{equation} \label{eqn:hodge}
D''=\dbar_E+\theta\quad ,\quad D'=\partial_E+\theta^\ast
    \end{equation}
    We wish to prove that $G_D=(D'')^2=0$, for then $\dbar_E$ is
    integrable, $\dbar_E\theta=0$, and $\theta\wedge\theta=0$;
    i.e.\ $(\dbar_E,\theta)$ is a Higgs bundle.

    Flatness of $D$ implies,
    \begin{enumerate}
        \item $\partial_E\theta=0$;
        \item $(\dbar_E\theta)^\ast=-\dbar_E\theta$;
        \item $\partial_E^2+\tfrac{1}{2}[\theta,\theta]=0$;
    \end{enumerate}
    By \eqref{eqn:harmonic}, we have
    $(\dbar_E\theta-(\dbar_E\theta)^\ast)\wedge\omega^{n-1}=0$, and
    combining this with (2) we have
    \begin{equation} \label{eqn:theta-primitive}
        \dbar_E\theta\wedge\omega^{n-1}=\dbar_E\theta\wedge\omega_0\wedge\Omega_0=0
    \end{equation}
    i.e.\ $G_D$ is primitive
    (note that we have only used the balanced condition for this part).
    
    To prove that $G_D=0$, we argue as in the proof of Theorem
    \ref{thm:siu-sampson} (see also \cite[proof of Thm.\ 5.1]{Corlette:88}). We have:
    \begin{align*}
        d\tr(\dbar_E\theta\wedge\theta^\ast)\wedge\Omega_0
        &=\partial\tr(\dbar_E\theta\wedge\theta^\ast)\wedge\Omega_0 \\
        &=\tr(\partial_E\dbar_E\theta\wedge\theta^\ast)\wedge\Omega_0+\tr(\dbar_E\theta\wedge(\dbar_E\theta)^\ast)\wedge\Omega_0\\
        &=-\frac{1}{2}\tr([[\theta,\theta^\ast],\theta]\wedge\theta^\ast)\wedge\Omega_0
        +\tr(\dbar_E\theta\wedge(\dbar_E\theta)^\ast)\wedge\Omega_0\\
        &=-\frac{1}{4}\tr([\theta,\theta]\wedge[\theta,\theta]^\ast)\wedge\Omega_0
        +\tr(\dbar_E\theta\wedge(\dbar_E\theta)^\ast)\wedge\Omega_0
    \end{align*}
so integrating,
$$
0=-\frac{1}{4}\int_X\tr([\theta,\theta]\wedge[\theta,\theta]^\ast)\wedge\Omega_0
        +\int_X\tr(\dbar_E\theta\wedge(\dbar_E\theta)^\ast)\wedge\Omega_0
$$
By the Hodge-Riemann property of $\Omega_0$, both terms on the right hand
side are nonpositive, and hence vanish. We conclude that $\dbar_E\theta=0$,
and $[\theta,\theta]=0$. By  (3) above, $\dbar_E$ is integrable, and this
completes the proof. 
\end{proof}

\section{Rigidity of representations of fundamental groups}\label{Section: rigidity}
For the sake of completeness, 
in this last  section we point out that two important results of Corlette
and Simpson generalize to our setting. 
Let $G_{\R}$ be a simple real algebraic group acting by isometries on the
irreducible bounded symmetric domain $G_{\R}/K$. We assume $(X, \omega)$ is 
a compact complex manifold with a balanced metric of Hodge-Riemann type.
Let $P$ be the principle $G_{\R}$ bundle with structure
group reduced to $K$. As in \cite{Corlette:91}, one can associate a volume
$\vol(P)$ to $P$ by defining it as a power of the first Chern class of $P$
up to a conformal factor. Now the following generalizes 
\cite[Thm.\ 0.1]{Corlette:91}.
 
\begin{thm}
Suppose $P$ is flat with $\vol(P)\neq 0$ and $G_{\R}/K$ is not of the form
    $\U(n,1)/\U(n)\times \U(1)$ or $\SO(2n+1,
    2)/\mathsf{S}(\mathsf{O}(2n+1)\oplus \mathsf{O}(2))$. Then the monodromy homomorphism of the fundamental group of $X$ into $G_{\R}$ is locally rigid as a homomorphism of the fundamental group of $X$ into the complexification of $G_{\R}$.  
\end{thm}

The argument follows by replacing  \cite[Prop.\ 2.4]{Corlette:91}
with argument in the proof of  Theorem \ref{thm:nah}  to get the
holomorphic property of the harmonic sections. This is the only place
where the K\"ahler assumption is needed in \cite{Corlette:91}. 
Simpson's argument in the K\"ahler
case also gives (see \cite{Simpson:92})
\begin{thm}
Suppose $\rho: \pi_1(X) \rightarrow \GL(n, \C)$ is a locally rigid representation
    of the fundamental group of $X$. Then the associated flat vector bundle is 
    the underlying vector bundle of a complex variation of Hodge structure. 
\end{thm}

\bibliography{../papers}

\end{document}